\documentclass[12pt]{amsart}%[cup7b]{cupbook}

\usepackage{amscd}
\usepackage{amsmath}
\usepackage{amsfonts}
\usepackage{amssymb}
\usepackage[mathscr]{eucal}
\usepackage[T1]{fontenc} 

\newtheorem{theorem}{Theorem}

\newtheorem{corollary}{Corollary} 
\newtheorem{definition}{Definition}
 
\newtheorem{lemma}{Lemma}

\newtheorem{proposition}{Proposition}
\newtheorem{remark}{Remark}
\newtheorem{remarks}[remark]{Remarks}  
\newtheorem{thmA}{Theorem}

 \def\cat{{\rm{CAT}}$(0)$ }

\def\T{\mathcal T}
\def\ssm{\smallsetminus}

\def\R{\mathbb R}

\def\E{\mathbb E}
\def\Si{\Sigma}

\def\fix{\hbox{{\rm{Fix}}}}

\def\L{\Lambda}

\def\<{\langle}
\def\>{\rangle}

\def\-{\underline}

\def\Z{\mathbb Z}

\def\S{\Sigma}

\def\G{\Gamma}

\def\g{\gamma}

\def\isom{\text{\rm{Isom}}}

\def\ball{\text{\rm{Ball}}}
\def\-{\overline}

\begin{document}

\title[Semisimple actions of mapping class groups on CAT$(0)$ spaces]
{Semisimple actions of mapping class groups on CAT$(0)$ spaces} 

\author[Martin R. Bridson]{Martin R.~Bridson}
\address{Mathematical Institute,
24-29 St Giles',
Oxford OX1 3LB, U.K. }
\email{bridson@maths.ox.ac.uk}

\thanks{This research was supported by a Senior
Fellowship from the EPSRC and a Royal
Society Wolfson Research Merit Award.
To appear in "The Geometry of Riemann Surfaces", LMS Lecture Notes 368.}

\begin{abstract} Let $\S$ be an orientable surface of finite type
and let ${\rm{Mod}}(\Sigma)$ be its mapping class group. We consider actions of ${\rm{Mod}}(\Sigma)$ by semisimple
isometries on complete {\rm{CAT}}$(0)$ spaces. If the genus of $\S$ is at least $3$, then in
any such action all Dehn twists act as elliptic isometries.
The action of ${\rm{Mod}}(\Sigma)$ on the completion of Teichm\"uller space
with the Weil-Petersson metric shows that there are interesting actions of this type.
Whenever the mapping class group of a
closed orientable surface of genus $g$ acts by semisimple isometries
on a complete {\rm{CAT}}$(0)$ space of dimension less than $g$ it must fix a point.
The mapping class group of a closed surface of genus $2$ acts properly by semisimple isometries
on a complete {\rm{CAT}}$(0)$ space of dimension $18$.
\end{abstract}

\maketitle

\section{Introduction} 
This article concerns actions of mapping class groups
by isometries on complete {\rm{CAT}}$(0)$ spaces. It records the contents of the lecture that I gave at Bill Harvey's 65th
birthday conference at Anogia, Crete in July 2007.

A {\rm{CAT}}$(0)$ space is a
geodesic metric space in which each geodesic triangle is no fatter than a triangle in the
Euclidean plane that has the same edge lengths (see Definition \ref{d:cat}). Classical
examples include complete 1-connected Riemannian manifolds with non-positive
sectional curvature and metric trees. The isometries of a {\rm{CAT}}$(0)$ space $X$ divide naturally into
two classes: the {\em{semisimple}} isometries are those for which there exists $x_0\in X$ such that
$d(\gamma.x_0, \, x_0) = |\g|$ where $|\g|:=\inf\{d(\gamma.y, \, y) \mid y\in X\}$; the remaining
isometries are said to be
{\em{parabolic}}. Semisimple isometries are further divided into {\em{hyperbolics}}, for which $|\g|>0$,
and {\em{elliptics}}, which have fixed points. Parabolics can be divided
into {\em{neutral parabolics}}, for which $|\g|=0$, and non-neutral parabolics.
If $X$ is a polyhedral space with only finitely many isometry types of cells, then all isometries
of $X$ are semisimple \cite{mrb:ss}.

E.~Cartan \cite{cartan} proved that the natural metric on a symmetric space of non-compact
type has non-positive sectional curvature. This gives an action of the mapping 
class group on a complete {\rm{CAT}}$(0)$ space:   
the morphism ${\rm{Mod}}(\Sigma_g)\to {\rm{Sp}}(2g,\Z)$ induced by the
action of ${\rm{Mod}}(\Sigma_g)$ on the first homology of $\S_g$, the closed orientable
surface of genus $g$,
gives an action of ${\rm{Mod}}(\Sigma_g)$ by isometries
on the symmetric space   for the symplectic group ${\rm{Sp}}(2g,\R)$.
In this
action,
 the Dehn twists in non-separating curves act as neutral parabolics.

In the  fruitful analogy between mapping class groups and lattices
in semisimple Lie groups,
Teichm\"uller space takes the role of the symmetric space. Unfortunately,
the Teichm\"uller metric does not have non-positive curvature \cite{masur1}.
On the other hand, the Weil-Petersson metric, although not complete,
 does have non-positive curvature
 \cite{wolp75, wolp87}. Since
the completion of a {\rm{CAT}}$(0)$ space is again a {\rm{CAT}}$(0)$ space \cite{BH} p.187,
it seems natural to complete the Weil-Petersson metric and to examine the action of the
mapping class group on the
completion   $\overline \T$  in order to elucidate the structure of the group.

\begin{thmA}
Let $\S$ be an orientable surface of finite type with negative
euler characteristic and empty boundary. The action of ${\rm{Mod}}(\Sigma)$ on the completion
of Teichm\"uller space in the Weil-Petersson metric is by semisimple
isometries. All Dehn twists act as elliptic isometries (i.e.~have fixed
points).
\end{thmA}

This theorem is a restatement of results in the literature
but I wanted to highlight it  as
 a motivating example.
 The essential points in the proof are described in Section \ref{s:WP}.
The fact that the Dehn twists have fixed points in the action on $\-\T$
is a manifestation of a general phenomenon:

\begin{thmA}\label{t:B} Let $\S$ be an orientable surface of finite type with genus
$g\ge 3$. Whenever ${\rm{Mod}}(\Sigma)$ acts by semisimple isometries on a complete
{\rm{CAT}}$(0)$ space, all Dehn twists in ${\rm{Mod}}(\Sigma)$ act as elliptic isometries
(i.e.~have fixed points).
\end{thmA}

One proves this theorem by comparing the
centralizers of Dehn twists in mapping class groups
with the  centralizers of
hyperbolic elements in isometry groups of {\rm{CAT}}$(0)$ spaces --- see Section \ref{s:B}.
The action of ${\rm{Mod}}(\Sigma)$ on the symmetric space for
${\rm{Sp}}(2g,\R)$
shows that one must weaken the conclusion of Theorem \ref{t:B} if one wants to drop the
hypothesis that ${\rm{Mod}}(\Sigma)$ is acting by semisimple isometries. The appropriate conclusion is that the Dehn twists act either as elliptics or
as neutral parabolics (see Theorem \ref{t:paras}).

Theorem \ref{t:B} provides information about actions of finite-index
subgroups of mapping class groups. For if $H$ is a subgroup of index $n$
acting by semisimple isometries on a {\rm{CAT}}$(0)$ space $X$, then the induced
action of ${\rm{Mod}}(\Sigma)$ on $X^n$ is again by semisimple isometries:
$\g\in{\rm{Mod}}(\Sigma)$ will be elliptic (resp.~hyperbolic) in the induced action
if and only if
any power of $\g$ that lies in $H$ was elliptic (resp.~hyperbolic) in the original action (cf.~remark \ref{r:abelian}).

The situation for genus 2 surfaces is quite different, as I shall explain
in Section \ref{s:genus2}.

\begin{thmA}\label{t:M_2 acts}
The mapping class group of a closed orientable surface of genus $2$ acts properly by
semisimple isometries on a complete \cat space of dimension
$18$.
\end{thmA}

The properness of the action in Theorem C contrasts sharply
with the nature of the actions in Theorems A, B and D.

I do not know if the mapping class group of a surface of genus $g>1$
can admit an action by semisimple
isometries, without a
global fixed point, on a complete {\rm{CAT}}$(0)$ space whose dimension is less than
that of the Teichm\"uller space. However, one can give a lower bound
on this dimension that is linear in $g$ (cf.~Questions 13 and 14
of \cite{BV} and Problem 6.1 in \cite{farb}).
To avoid complications, I shall state this only in the closed case.

\begin{thmA}\label{t:D}
Whenever the mapping class group of a
closed orientable surface of genus $g$ acts by semisimple isometries
on a complete {\rm{CAT}}$(0)$ space of dimension less than $g$ it fixes a point.
\end{thmA}

Here, "dimension'' means topological covering dimension. An
outline of the proof is given in Section \ref{s:g-bound}; the details are given in \cite{mrb:g-bound}. The  strategy of proof is based on
the "ample duplication" criterion in \cite{mrb:helly}.
The semisimple hypothesis can be weakened:
it is sufficient to assume that there are no non-neutral parabolics, cf.~Theorem \ref{t:paras}.

I make no claim regarding the sharpness of
the dimension bounds in  Theorems \ref{t:M_2 acts} and \ref{t:D}.
\bigskip

The lecture on which this paper
is based was a refinement
of a lecture that I gave on 8 December 2000 in
Bill Harvey's seminar at King's College London.
Bill worked tirelessly over many years to maintain a  geometry  and
topology seminar
in London. Throughout that time he shared many insights with visiting
researchers and always entertained them generously. His mathematical
writings display the same generosity of spirit: he has written clearly and
openly about his ideas rather than hoarding them until some arcane
goal was achieved.  The  benefits
of this openness are most clear
in his highly prescient and influential papers
introducing the curve complex \cite{bill, bill2}.

The book \cite{bill-book}
on discrete groups
and automorphic forms that Bill edited in 1977 had a great
influence on me when I was a graduate study.
At a more personal level, he and his wife
Michele have been immensely kind to me and my family over
many years. It is
therefore with the greatest pleasure that I dedicate these
observations about the mapping class group to him on the
occasion of his sixty fifth birthday.

\section{Centralizers, fixed points and Theorem B}\label{s:B}

\begin{definition}\label{d:cat} Let $X$ be a  geodesic metric space.
A geodesic triangle $\Delta$ in $X$ consists of three points $a,b,c\in X$ and
three geodesics $[a,b],\, [b,c],\, [c,a]$. Let $\-\Delta\subset\E^2$ be a triangle
in the Euclidean plane with the same edge lengths as $\Delta$
and let $\overline x\mapsto
x$ denote the map  $\-\Delta\to\Delta$
that sends each side of $\-\Delta$ isometrically
onto the corresponding side of $\Delta$.
One says that $X$ is a
{\rm{CAT}}$(0)$ space if for all $\Delta$ and all $\-x,\-y\in\-\Delta$ the inequality
 $d_X(x,y)\le d_{\E^2}(\-x, \-y)$ holds.

Note that in a {\rm{CAT}}$(0)$ space there is a unique geodesic $[x,y]$ joining
each pair of points $x,y\in X$.
A subspace $Y\subset X$ is said to be {\em{convex}} if
$[y,y']\subset Y$ whenever $y,y'\in Y$.
\end{definition}

\begin{lemma}\label{l:axis} If $X$ is a complete {\rm{CAT}}$(0)$ space
then an isometry $\gamma$ of $X$ is hyperbolic if and only if $|\gamma|>0$
and there is a $\gamma$-invariant convex subspace of $X$ isometric
to $\R$. (Each such subspace is called an axis for $\gamma$.)
\end{lemma}

\begin{proposition}\label{l:hyp} Let $\G$ be a group acting by
isometries on a complete {\rm{CAT}}$(0)$ space $X$. If $\gamma\in\G$
acts as a hyperbolic isometry then
$\gamma$ has infinite order in the abelianisation of
its centralizer $Z_\G(\gamma)$.
\end{proposition}

\begin{proof} This is proved on page 234 of \cite{BH} (remark 6.13).
The main points
are these:
 if $\gamma$ is hyperbolic then the union of
the axes of $\gamma$ splits isometrically
as $Y\times\R$; this subspace and its splitting are
preserved
by $Z_\G(\gamma)$; the action on the
second factor gives a homomorphism from $Z_\G(\g)$ to the
abelian group ${\rm{Isom}}_+(\R)$, in which the image of $\g$
is non-trivial.
\end{proof}

In the light of this proposition, in order to prove Theorem \ref{t:B} it
suffices to show that if $\S$ is a surface of finite type with
genus $g\ge 3$, then the Dehn twist $T$ about any simple closed
curve $c$ in $\S$ does not have infinite order in the abelianisation
of its centralizer.

\begin{proposition}If $\S$ is an orientable surface of finite type that has
genus at least $3$ (with any number of boundary components and
punctures) and if $T$ is the Dehn twist about any simple
closed curve $c$ in $\S$, then the abelianisation of the centralizer of
$T$ in ${\rm{Mod}}(\Sigma)$ is finite.
\end{proposition}

\begin{proof}
The centralizer of $T$ in ${\rm{Mod}}(\Sigma)$
 consists of mapping classes of homeomorphisms that leave
$c$ invariant. This is a homomorphic image of the mapping class group
of the surface obtained by cutting $\S$ along $c$.
(This surface has two boundary components corresponding to $c$
and hence two Dehn twists mapping to $T$.) Since $\S$ has genus
at least $3$, at least one component of the cut-open surface has genus $g\ge 2$. The
mapping class group of such a surface  has finite
abelianisation ---  see
\cite{korkmaz} for a concise survey and references.
\end{proof}

\begin{remark}\label{r:abelian}
As we remarked in the introduction, Theorem \ref{t:B} gives restrictions
on how subgroups of finite index in ${\rm{Mod}}(\Sigma)$ can act on
{\rm{CAT}}$(0)$ spaces.
For example, given an orientable surface of genus $g\ge 3$
and a homomorphism $\phi$ from a subgroup $H<
{\rm{Mod}}(\Sigma)$ of  index $n$
to a group $G$ that acts by hyperbolic isometries on
a complete {\rm{CAT}}$(0)$ space $X$, we can apply Theorem \ref{t:B} to the
induced action\footnote{One can regard
this as "multiplicative induction" in the sense of \cite{tD}
p.~35: if $H<G$
has index $n$ and acts on $X$ then one identifies $X^n$
with the space of $H$-equivariant maps $f:G\to X$ and considers the (right)
action $(g.f)(\gamma) := f(\gamma g)$; a power of each $g\in G$ preserves
the factors of $X^n$, so it follows from \cite{BH} p.~231--232 that the action
of $G$ is by semisimple isometries if the action of $H$ is.}
 of ${\rm{Mod}}(\Sigma)$ on $X^n$ and
hence deduce that any power of a Dehn
twist that lies in $H$ must lie in the kernel of $\phi$. Taking $G=\Z$
tells us that powers of Dehn twists cannot have infinite image in the
abelianisation of $H$. A more explicit proof of this last fact was given recently
by Andew Putman \cite{andy}.
\end{remark}

I am grateful to Pierre-Emmanuel Caprace, Dawid Kielak, Anders Karlsson and Nicolas Monod
for  their comments
concerning the following extension of Theorem \ref{t:B}.

\begin{theorem}\label{t:paras}
Let $\S$ be an orientable surface of finite type with genus
$g\ge 3$. Whenever ${\rm{Mod}}(\Sigma)$ acts by  isometries on a complete
{\rm{CAT}}$(0)$ space, each Dehn twist $T\in {\rm{Mod}}(\Sigma)$ acts either as an elliptic isometry
or as a neutral parabolic
(i.e.~$|T|=0$).
\end{theorem}

\begin{proof} The proof of Theorem \ref{t:B} will apply provided we can extend
Proposition \ref{l:hyp} to cover non-neutral parabolics. In order to
appreciate this extension, the reader should be familiar with the
basic theory of Busemann functions in {\rm{CAT}}$(0)$ spaces, \cite{BH}  Chap.~II.8.

If $\g$ is a parabolic
isometry with $|\g|>0$, then a special case of a result
of Karlsson and Margulis \cite{KM} (cf.~\cite{Karl}, p.~285) shows
that $\g$ has a unique fixed point at infinity  $\xi\in\partial X$ 
with the property that  $\frac 1 n d(\g^n.x, c(n|\g|))
\to 0$ as $n\to \infty$ for every $x\in X$ and every geodesic ray $c:[0,\infty)\to X$
with $c(\infty)=\xi$. Now, $Z_\G(\g)$ fixes $\xi$ and 
acts  on any Busemann
function centred at $\xi$ by the formula $z.\beta(t) = \beta(t)+\phi(z)$,
where $\phi : Z_\G(\g)\to\R$ is a homomorphism.  Since
$\phi(\g)=-|\g|$, this is only possible if $\gamma$ has
infinite order in the abelianisation of $Z_\G(\gamma)$.
\end{proof}

I.~Kapovich and B.~Leeb \cite{KL} were the first to prove that if $g\ge 3$ then
 ${\rm{Mod}}(\Sigma_g)$ cannot act properly by semisimple isometries on a complete
{\rm{CAT}}$(0)$ space; cf.~\cite{BH} p.~257.

\section{Augmented Teichm\"uller space and Theorem A}\label{s:WP}

Let $\Sigma$ be an orientable hyperbolic surface of finite
type with empty boundary and let $\-\T$ denote the completion of
its Teichm\"uller space  equipped with the Weil-Petersson metric.
$\-\T$ is equivariantly homeomorphic to the augmented Teichm\"uller space
defined by Abikoff \cite{abi}. Wolpert's concise survey \cite{wolp} provides
a clear introduction and ample references to the facts that we need
here.

Masur \cite{masur2} describes the
metric structure of   $\-\T\ssm \T$ as follows. It is a union of
strata $\T_C$ corresponding to the homotopy classes of
systems $C$ of disjoint, non-parallel, simple
closed curves on $\S$. The stratum corresponding to $C$ is the
Teichm\"uller space
of the nodal surface obtained by shrinking each loop $c\in C$
to a pair of cusps (punctures); more explicitly, it is the product of
Teichm\"uller spaces for the components of $\S\ssm \bigcup C$
(with punctures in place of the pinched curves),
each equipped with its Weil-Petersson metric.  The identification
of $\T_C$ with this product of Teichm\"uller spaces is equivariant
with respect to the natural map from the subgroup of ${\rm{Mod}}(\S)$
that preserves $C$ and its components
to the mapping class group of the nodal surface.
Importantly,
$\T_C\subset\-\T$ is a convex subspace \cite{DW}.

The Dehn twists (and hence multi-twists $\mu$)
in the curves of $C$ act trivially on $\T_C$ and hence are elliptic
isometries of $\-\T$.

Daskalopoulos and Wentworth \cite{DW} proved that every
pseudo-Anosov element $\psi$ of ${\rm{Mod}}(\S)$ acts as a hyperbolic
isometry of $\-\T$; indeed
 each has an axis contained in $\T$. If $\rho\in{\rm{Mod}}(\Sigma)$ leaves invariant
a curve system $C$ and each of the components of $\S\ssm C$, then
either it is a multi-twist (and hence acts as an elliptic isometry of
$\-\T$) or else it
acts as a pseudo-Anosov on one of the  components of $\S\ssm C$. In
the latter case $\rho$ will act as a hyperbolic isometry of $\T_C$,
and hence of $\-\T$, since
$\T_C\subset\-\T$ is convex. An isometry of a complete {\rm{CAT}}$(0)$ space is hyperbolic
(resp.~elliptic) if and only if every proper power of it is hyperbolic (resp.~elliptic)
\cite{BH}, p.~231--232.
Every element of the mapping class group has a proper power that
is one of the three types $\mu, \psi, \rho$ considered above \cite{wpt}.
Thus Theorem A is proved.

\begin{remarks}\label{rems}(1) I want to emphasize once again that I stated Theorem A only
to provide context: nothing in the proof is original. Moreover, since
I first wrote this note, Ursula Hamenstadt has given essentially the same proof in \cite{ursula}.

(2)
Masur-Wolf \cite{MW} and Brock-Margalit \cite{BM} have shown that
${\rm{Mod}}(\S)$ is the full isometry group of $\-\T$. I am grateful
to Jeff Brock for a helpful correspondence on this point.

(3) In the  proof of Theorem A, the roots of multitwists emerged as the only elliptic isometries
(provided that we regard the identity as a multitwist).
Combining this observation with Theorem \ref{t:B}, we see that any homomorphism from the mapping
class group of a surface of genus $g\ge 3$ to another mapping class group must send roots of
multitwists to roots of multitwists.
 This contrasts with the fact that there are injective homomorphisms
between mapping class groups of once-punctured surfaces of higher genus
that send pseudo-Anosov elements to multitwists \cite{ALS}.
\end{remarks}

\section{Criteria for common fixed points}

The classical theorem of Helly concerns the combinatorics
of families of
convex subsets in $\R^n$. There are many
variations on this theorem in the literature. For
our purposes the following will be sufficient (see
\cite{mrb:helly}, \cite{farb-helly} and references therein).

\begin{proposition}\label{t:helly} Let $X$ be a complete
\cat space of (topological covering) dimension at most
$n$, and let $C_1,\dots,C_N\subset X$ be closed
convex subsets. If every $(n+1)$ of the $C_i$
have a point of intersection, then $\bigcap_{i=1}^NC_i\neq\emptyset$.
\end{proposition}

When applied to the fixed point sets $C_i=\fix(s_i)$ with $s_i\in S$, this
implies:

\begin{corollary}\label{c:basic} Let $\G$ be a group acting by isometries on a complete
{\rm{CAT}}$(0)$ space $X$ of dimension at most $n$ and suppose that $\G$
is generated by the finite set $S$. If every $(n+1)$-element subset of
$S$ fixes a point of $X$, then $\G$ has a fixed point in $X$.
\end{corollary}

We shall need a refinement of this result that relies on the
following well-known proposition, \cite{BH} p.179.
We write $\isom(X)$ for the group of isometries of a metric space $X$
and $\ball_r(x)$ for the closed ball of radius $r>0$ about $x\in X$.
Given a subspace $Y\subseteq X$, let $r_Y:=\inf\{r\mid 
Y\subseteq \ball_r(x), \text{ some } x\in X\}$.

\begin{proposition}
If $X$ is a complete \cat space and $Y$ is a  bounded
subset, then there is a unique point $c_Y\in X$ such that
$Y\subseteq \ball_r(c_Y)$.
\end{proposition}

\begin{corollary} Let $X$ be a complete \cat space. If $H<\isom(X)$ has
a bounded orbit then $H$ has a fixed point.
\end{corollary}

\begin{proof} The centre $c_O$ 
of any $H$-orbit $O$ will be a fixed
point.
\end{proof}

\begin{corollary}\label{c:commutes}
 Let $X$ be a complete  {\rm{CAT}}$(0)$ space. If
the groups $H_1,\dots, H_\ell<\isom(X)$ commute and
$\fix(H_i)$ is non-empty for $i=1,\dots,\ell$, then
$\bigcap_{i=1}^\ell \fix(H_i)$ is non-empty.
\end{corollary}

\begin{proof} A simple induction reduces us to the case
$\ell=2$. Since $\fix(H_2)$ is non-empty, each $H_2$-orbit
is bounded. As $H_1$ and $H_2$ commute, $\fix(H_1)$ is
$H_2$-invariant and therefore contains an $H_2$-orbit. As
$\fix(H_1)$ is convex, the centre of this (bounded) orbit
is also in $\fix(H_1)$, and therefore is fixed by $H_1\cup H_2$.
\end{proof}

Building on these elementary observations, one can prove the following; see \cite{mrb:helly}.

\begin{proposition}[Bootstrap Lemma]\label{l:bootstrap}
Let $k_1,\dots,k_n$ be positive
integers and let $X$ be a complete \cat space of dimension less than $k_1+\dots+k_n$.
Let $S_1,\dots,S_n \subset\isom(X)$ be subsets with $[s_i,s_j]=1$
for all $s_i\in S_i$ and $s_j\in S_j\ (i\neq j)$.

If, for $i=1,\dots,n$, each $k_i$-element subset of $S_i$ has a
fixed point in $X$, then for some $i$ every
finite subset of   $S_i$ has a fixed point.
\end{proposition}

When applying the Bootstrap Lemma one has to
overcome the fact that the conclusion only applies
to {\em{some}} $S_i$.
 A convenient way of gaining more control
 is to restrict attention to conjugate sets.

\begin{corollary}[Conjugate Bootstrap]\label{c:conjug}
Let $k$ and $n$ be positive
integers and let $X$ be a complete \cat space of dimension less than $nk$.
Let $S_1,\dots,S_n$ be conjugates of a subset
$S\subset\isom(X)$ with $[s_i,s_j]=1$
for all $s_i\in S_i$ and $s_j\in S_j\ (i\neq j)$.

If each $k$-element subset of $S$ has a
fixed point in $X$, then so does each finite subset
of $S$.
\end{corollary}

\section{Some surface topology}\label{s:g-bound}

The reader will  recall that, given a
closed orientable surface $\Si$ and two compact
homeomorphic
sub-surfaces with boundary $T,T'\subset \Si$,
there exists an automorphism of $\Si$ taking $T$
to $T'$ if and only
$\Si\ssm T$ and
$\Si\ssm T'$ are homeomorphic. In particular,
two homeomorphic sub-surfaces are in the
same orbit under the action of ${\rm{Homeo}}(\Si)$ if the
complement of each is connected.

The relevance of this observation to our purposes
is explained by the following lemma, which will be used in tandem with the Conjugate Bootstrap.

\begin{lemma}\label{l:disjoint}  Let $H$ be the
subgroup of ${\rm{Mod}}(\Sigma)$ generated by the Dehn
twists in a set $C$ of loops
 all of which are
contained in a compact sub-surface $T\subset\Si$
with connected complement.
If $\Si$ contains $m$ mutually disjoint sub-surfaces
$T_i$ homeomorphic to $T$, each with connected
complement, then ${\rm{Mod}}(\Sigma)$  contains $m$
mutually-commuting conjugates $H_i$ of $H$.
\end{lemma}

\begin{proof}
Since the complement  of each
$T_i$ is connected,
there is a homeomorphism $\phi_{i}$ of $\Si$ carrying
$T$ to $T_i$. Define $H_i$ to be the subgroup
of ${\rm{Mod}}(\Sigma)$  generated by the Dehn twists in the
loops $\phi_i(C)$. Since the various
$H_i$ are supported in disjoint sub-surfaces, they
commute.
\end{proof}

\subsection{The Lickorish generators}

\def\L{{\rm{Lick}}}
Raymond Lickorish \cite{lick}
proved that the mapping class
group of a closed orientable surface of genus $g$
is generated by the Dehn twists in  $3g-1$ non-separating
loops, each pair of which
intersects in at most one point.
Let $\L$ denote this set of loops.

We say that a subset
$S\subset \L$ is {\em{connected}} if the union
$U(S)$ of the
loops in $S$ is connected. An analysis of $\L$
reveals the following fact, whose proof is deferred to \cite{mrb:g-bound}.
In this statement all sub-surfaces are assumed to be compact.

\begin{proposition}\label{p:lick}
Let $S\subset \L$ be a connected
subset.
\begin{enumerate}
\item
If $|S|=2\ell$ is even, then $U(S)$ is either contained
in a  sub-surface of genus  $\ell$ with $1$ boundary
component, or else  in a non-separating sub-surface
of genus at most $\ell-1$ with $3$ boundary components.

\item
If $|S|=2\ell+1$ is odd, then $U(S)$ is either contained
in a non-separating subsurface of genus $\ell$ with
at most $2$ boundary components, or else
in a non-separating sub-surface of genus
at most $(\ell-1)$ that has at most $3$ boundary components.
\end{enumerate}
\end{proposition}

 \subsection{The Proof of Theorem D: an outline}

We must argue that when  the mapping class group of a
closed orientable surface of genus $g$ acts
without neutral parabolics
on a complete {\rm{CAT}}$(0)$ space $X$ of dimension less than $g$ it must fix a point.

The case $g=1$ is trivial. A complete \cat\ space
of dimension 1 is an $\mathbb R$-tree, so for $g=2$ the assertion
of the theorem is that the mapping class group of a genus 2
surface has property ${\rm{F}}\R$. This was proved by
Culler and Vogtmann \cite{CV}.

Assume $g\ge 3$. According to Corollary \ref{c:basic}, we will be
done if we can show that each subset $S\subset\L$ with $|S|\le g$
has a fixed point in $X$. We proceed by induction on $|S|$.
 Theorem \ref{t:paras} covers the base case $|S|=1$.

If $S$ is not connected, say
$S=S_1\cup S_2$ with $U(S_1)\cap U(S_2)=\emptyset$, then
the subgroups $\langle S_1\rangle$ and $\langle
S_2\rangle$
commute. Each has a
fixed point since $|S_i|<|S|$, so
Corollary \ref{c:commutes} tells us that $S$
has a fixed point.

Suppose now that $S$ is connected.
If $|S|=2\ell$ is even then Proposition \ref{p:lick} tells us
that $U(S)$ is contained either in a sub-surface of genus
$\ell$ with $1$ boundary component or else in a
non-separating  sub-surface of genus
$\ell-1$ with $3$ boundary components.
In either case
 one can fit $\lfloor g/\ell\rfloor$ disjoint
copies of this sub-surface into $\Si_g$.
Lemma \ref{l:disjoint} then provides us with $\lfloor g/\ell\rfloor$
mutually-commuting conjugates of $\langle S\rangle$.
As all proper subsets of $S$ are assumed to have a fixed point
and the dimension of
$X$ is  less than $(2\ell -1)\lfloor g/\ell\rfloor$,
the Conjugate Bootstrap (Corollary \ref{c:conjug}) tells us
that $S$ will have a fixed point.

The argument for $|S|$ odd is similar.
For  a detailed proof, see \cite{mrb:g-bound}.
\hfill$\square$

\section{A proper semisimple action of ${\rm{Mod}}(\Si_2)$}\label{s:genus2}

\begin{theorem}
The mapping class group of a closed surface of genus $2$ acts properly by
semisimple isometries on a complete \cat space of dimension
$18$.
\end{theorem}

\def\M{{\rm{Mod}}}

\begin{proof}
The hyperelliptic involution $\tau$ is central in $\M(\Si_2)$.
The quotient orbifold $\Si_2/\<\tau\>$ is a sphere with $6$
marked points and the
action of $\M(\Si_2)$ on this quotient
 induces a homomorphism $\M(\Si_2)\to \M(\Si_{0,6})$. This
 is onto \cite{BiHi} and the kernel is $\langle\tau\rangle$.
Thus we have a short exact
sequence
$$
1\to \Z_2 \to \M(\Si_2)\to  \M(\Si_{0,6})\to 1.
$$

For each positive integer $n\ge 2$ there is
a natural homomorphism from the braid group $B_n$
to $\M(\Si_{0,{n+1}})$; the image of this map is of index
$(n+1)$ and the kernel
is the centre of $B_n$, which is infinite cyclic.
This map
admits the following geometric interpretation. Regard $B_n$ as the
mapping class group of the $n$ punctured disc $D$
with $1$ boundary
component. One maps $B_n$ to $\M(\Si_{0,{n+1}})$ by
attaching the boundary of a once-punctured disc to $\partial D$
and extending homeomorphisms of $D$ by the identity on the
attached disc. The image of $B_n$ is the subgroup of
$\M(\Si_{0,{n+1}})$ that stabilizes the puncture in the added disc;
this  has index $n+1$.
The centre of $B_n$ is the mapping class of
the Dehn twist $\zeta$ in a loop
parallel to the boundary $\partial D$. This twist
becomes trivial in $\M(\Si_{0,{n+1}})$, and  it
generates the kernel of $B_n\to\M(\Si_{0,{n+1}})$.
Thus we have a second short exact sequence
$$
1\to \Z \to B_n \to \G\to 1
$$ where $\G\subset \M(\Si_{0,{n+1}})$ is of index $n+1$.

Brady and McCammond \cite{brady} and independently
Krammer (unpublished), showed that $B_5$
is the fundamental group of a compact non-positively curved piecewise-Euclidean
complex $X$ of dimension $4$ that has no free faces.
It follows from \cite{BH} II.6.15(1) and II.6.16 that the
universal cover of $X$ splits isometrically
as a product $Y\times\R$ and that the
quotient of $B_5$ by its centre acts properly on $Y$
(II.6.10(4) {\em{loc. cit.}})
by semisimple isometries (II.6.9 {\em{loc. cit.}}).
By inducing, as in remark \ref{r:abelian},
 we obtain
an action of  $\M(\Si_{0,6})$ on $Y^6$ that is again proper and
semisimple. And since the kernel
of $\M(\Si_2)\to\M(\Si_{0,6})$ is finite, the resulting
action of $\M(\Si_2)$ on $Y^6$ is also proper.
\end{proof}


\begin{thebibliography}{9999 2000}


\bibitem[Abi 1977]{abi} W.~Abikoff, {\em Degenerating families of
Riemann surfaces}, Ann. of Math. {\bf 105} (1977), 29--44.


\bibitem[ALS]{ALS} J.~Aramayona, C.J.~Leininger and J.~Souto,
{\em Injections of mapping class groups}, preprint, arXiv:0811.0841.

\bibitem[BiHi 1971]{BiHi} J.S.~Birman and H.M.~Hilden,
{\em On the mapping class groups of closed surfaces as covering
spaces},
in "Advances in the theory of Riemann surfaces" (Stony Brook 1969),
pp.~81--115, Ann. of Math. Studies {\bf 66}, Princeton Univ. Press,
Princeton NJ, 1971.

\bibitem[BrMc]{brady} T.~Brady and J.~McCammond,
{\em A {\rm{CAT}}$(0)$ structure for the 5-string braid group},
 in preparation.

\bibitem[Bri 1999]{mrb:ss} M.R.~Bridson,
{\em On the semisimplicity of polyhedral isometries},
 Proc. Amer. Math. Soc.  {\bf 127}  (1999), 2143--2146.

\bibitem[Bri1]{mrb:helly} M.R.~Bridson, {\em Helly's theorem, {\rm{CAT}}$(0)$ spaces,
and actions of automorphism groups of free groups}, preprint,
Oxford, 2007.

\bibitem[Bri2]{mrb:g-bound} M.R.~Bridson, {\em On the dimension
of {\rm{CAT}}$(0)$ spaces where mapping class groups act},
preprint, Oxford, 2009.


\bibitem[BriH 1999]{BH} M.R.~Bridson and A.~Haefliger, "Metric spaces of
nonpositive curvature'', Grundlehren der Math. Wiss. \textbf{319},
Springer-Verlag, Berlin, 1999.

\bibitem[BriV 2006]{BV} M.R.~Bridson and K.~Vogtmann,
{\em Automorphism groups of free groups, surface groups and free abelian groups},
in "Problems on mapping class groups and related topics" (B.~Farb, ed.), pp.~301--316, Proc. Sympos. Pure Math., {\bf 74}, Amer. Math. Soc., Providence, RI, 2006.

\bibitem[BroM 2007]{BM} J.~Brock and D.~Margalit,
{\em Weil-Petersson isometries via the pants complex},
Proc. Amer. Math. Soc. {\bf 135} (2007), 795--803.



\bibitem[Ca 1926]{cartan} E.~Cartan, {\em Sur une classe remarquable
d'espaces de Riemann},
Bull. Soc. Math. France {\bf 54} (1926), 214--264.





\bibitem[CuVo 1996]{CV} M. Culler and K. Vogtmann,
{\em A group theoretic criterion for property FA},
Proc. Amer. Math. Soc. {\bf 124} (1996),  677--683.


\bibitem[DaWe 2003]{DW}
G.~Daskalopoulos and R.~Wentworth,
{\em Classification of Weil-Petersson isometries},
Amer. J. Math. {\bf 125} (2003), 941--975.

\bibitem[tomD 1987]{tD} T.~tom Dieck,
"Transformation groups", Studies in Mathematics {\bf 8},
de Gruyter, Berlin-New York, 1987.

\bibitem[Farb 2006]{farb} B.~Farb, {\em
Some problems on mapping class groups and moduli space},
in "Problems on mapping class groups and related topics" (B.~Farb, ed.),
pp.~11--55, Proc. Sympos. Pure Math., {\bf 74}, Amer. Math. Soc.,
Providence, RI, 2006.

\bibitem[Farb]{farb-helly} B.~Farb, {\em Group actions and Helly's
theorem}, preprint, arXiv:0806.1692.
 

\bibitem[Ham]{ursula} U.~Hamenstadt, {\em Dynamical properties of the
Weil-Petersson metric}, preprint, arXiv:0901.4301.


\bibitem[Harv 1977]{bill-book} W.J.~Harvey (ed.),
"Discrete groups and automorphic forms", Academic Press, London, 1977.


\bibitem[Harv 1979]{bill2} W.J.~Harvey,
{\em Geometric structure of surface mapping class groups}, in "Homological group
Theory" (Durham 1977), pp.~255-269.
London Math. Soc. Lecture Note Ser. {\bf 36},
Cambridge Univ. Press, Cambridge-New York, 1979.


\bibitem[Harv 1981]{bill} W.J.~Harvey,
{\em Boundary structure of the modular group}, in "Riemann surfaces and
related topics" (Stony Brook 1978), pp.~245--251, Ann. of Math. Stud., {\bf 97},
Princeton Univ. Press, Princeton NJ, 1981.

 

\bibitem[Iv 2002]{ivanov} N.V.~Ivanov, {\em Mapping class groups},
in "Handbook of geometric topology" (R.~Daverman, R.B.~Sher, eds.),
Elsevier Science, Amsterdam, 2002, pp.~523--633.

\bibitem[KL 1995]{KL}
M.~Kapovich and B.~Leeb, 
{\em Actions of discrete groups on nonpositively curved spaces},
Math. Ann. {\bf 30}6 (1996), no. 2, 341--352. 


\bibitem[Ka 2002]{Karl} A.~Karlsson, {\em Nonexpanding maps, Busemann
functions, and multiplicative ergodic theory}, in "Rigidity in dynamics and
geometry" (Cambridge, 2000), pp.~283--294. Springer-Verlag, Berlin, 2002.


\bibitem[KaMa 1999]{KM}
A.~Karlsson and G.~Margulis,
{\em A multiplicative ergodic theorem and nonpositively curved spaces},
Comm. Math. Phys. {\bf 208} (1999), 107--123.


\bibitem[Kor 2002]{korkmaz} M.~Korkmaz,
{\em Low-dimensional homology groups of mapping class groups: a survey},
Turk. J. Math {\bf 26} (2002), 101--114.

\bibitem[Lic 1964]{lick}
W.B.R.~Lickorish,
{\em A finite set of generators for the homeotopy group of a 2-manifold}, Proc. Cambridge Philos. Soc. 60 (1964), 769778.

\bibitem[Lin 1971]{linch} M.R.~Linch,
{\em
On metrics in Teichm\"uller spaces}, Ph.D. Thesis, Columbia Univ., New York, 1971.


\bibitem[Mas 1975]{masur1} H.~Masur,
{\em  On a class of geodesics in Teichm\"uller space},
Ann. of Math. (2) {\bf 102}  (1975),   205--221.

\bibitem[Mas 1976]{masur2} H.~Masur,
{\em  The extension of the Weil-Petersson metric to
the boundary of Teichm\"uller space},
Duke Math J. {\bf 43}  (1976),   623--635.


\bibitem[MasW 2002]{MW} H.~Masur and M.~Wolf,
{\em The Weil-Petersson isometry group},
Geom. Dedicata {\bf 93}  (2002),   177--190.


\bibitem[Put]{andy} A.~Putman,
{\em  A note on the abelianizations of finite-index subgroups of the mapping class group}, preprint arXiv:0812.0017 .

\bibitem[Thu 1988]{wpt} W.P.~Thurston,
{\em On the geometry and dynamics of diffeomorphisms of surfaces},
Bull. Amer. Math. Soc. {\bf 19} (1988), 417--431.

\bibitem[Wol 1975]{wolp75} S.A.~Wolpert, {\em Noncompleteness of the
Weil-Petersson metric for Teichm\"uller space}, Pacific J. Math.,
{\bf 61} (1975), 573--577.

\bibitem[Wol 1987]{wolp87} S.A.~Wolpert, {\em Geodesic length
functions and the Nielsen problem}, J. Diff. Geom.
{\bf 25} (1987), 275--296.

\bibitem[Wol 2006]{wolp} S.A.~Wolpert, {\em Weil-Petersson
perspectives},
in "Problems on mapping class groups and related topics" (B.~Farb, ed.), pp.~301--316, Proc. Sympos. Pure Math., {\bf 74}, Amer. Math. Soc., Providence, RI, 2006.


\end{thebibliography}
\end{document}